\documentclass{cimart}

\usepackage{epstopdf}

\title{
   Partial-dual polynomial as a framed weight system
    }

\author{
    Qingying Deng, Fengming Dong, Xian'an Jin and Qi Yan
    }

\authorinfo[
    Q. Deng]{
    School of Mathematics and Computational Science, Xiangtan University, P. R. China}{%
    qingying@xtu.edu.cn
    }

\authorinfo[
    F. Dong]{
    National Institute of Education, Nanyang Technological University, Singapore}{%
    fengming.dong@nie.edu.sg, donggraph@163.com
    }

\authorinfo[
    X. Jin]{
    School of Mathematical Sciences, Xiamen University, P. R. China}{%
    xajin@xmu.edu.cn
    }

\authorinfo[
    Q. Yan (Corresponding author)]{
    School of Mathematics and Statistics, Lanzhou University, P. R. China}{%
    yanq@lzu.edu.cn
    }

\abstract{%
    Recently, Chmutov proved that the partial-dual polynomial considered as a function on chord diagrams satisfies the four-term relations. In this paper, we show that this function on framed chord diagrams also satisfies the four-term relations, i.e., is a framed weight system.
    }



\VOLUME{31}
\YEAR{2023}
\ISSUE{3}
\firstpage{151}
\DOI{https://doi.org/10.46298/cm.13020}

\begin{document}

\section{Introduction}
Chmutov \cite{CG} introduced a far-reaching generalisation of geometric duality, called partial duality, and together with other partial dualities, it has received ever-increasing attention, and their applications span topological graph theory, knot theory, matroids/delta-matroids, and physics. In analogy with the extensively studied polynomial in topological graph theory that enumerates by Euler genus all embeddings of a given graph, Gross, Mansour and Tucker \cite{GMT} introduced the partial-dual polynomials for arbitrary ribbon graphs. Comprehensive surveys of the partial-dual polynomial may be found in \cite{CHVT, GMT, GMT2, YJ1, YJ2, YJ3}.

A chord diagram can be interpreted as an orientable ribbon graph with a single vertex. Chmutov \cite{CHS} proved that the partial-dual polynomial considered as a function on chord diagrams satisfies the four-term relations, thus it is a weight system from the theory of Vassiliev knot invariants. The main combinatorial object in our study is framed chord diagrams, i.e., chord diagrams with chords of two types. In this paper, we show that the partial-dual polynomial considered as a function on framed chord diagrams also satisfies the four-term relations, i.e., is a framed weight system.

\section{Framed chord diagrams}

\begin{definition}[\cite{IDP, SKL}]
A chord diagram is a cubic graph consisting of a selected oriented Hamiltonian cycle (the core circle) and several non-oriented edges (chords) connecting points on the core circle.  A chord diagram is framed if a map (a framing) from the set of chords to $\mathbb{Z}/ 2\mathbb{Z}$ is given, i.e., every chord is endowed with 0 or 1. The chords with framing 0 are said to be orientable, and those with framing 1 are said to be disorienting or twisted.
\end{definition}

We consider all framed chord diagrams up to orientation and framing preserving isomorphisms of graphs taking one core circle to the other one. We shall usually omit the orientation of the core circles in chord diagrams, assuming that it is oriented counterclockwise. For ease of description, chords having framing 0 are drawn in solid curves, and those having framing 1 are drawn in dashed ones. An example is given in Figure \ref{p06}.

In this paper, we will also use framed chord diagrams on several circles. Analogously, a framed chord diagram on $n$ circles is a cubic graph consisting of $n$ oriented disjoint circles (the core circles) and several edges (solid or dashed chords) connecting points on the core circles. Let $M^f$ be the free $\mathbb{Z}$-module generated by all framed chord diagrams. Each element of $M^f$ is a finite linear combination of framed chord diagrams with integer coefficients.

\begin{figure}[htbp!]
  \centering
  \includegraphics[width=14cm]{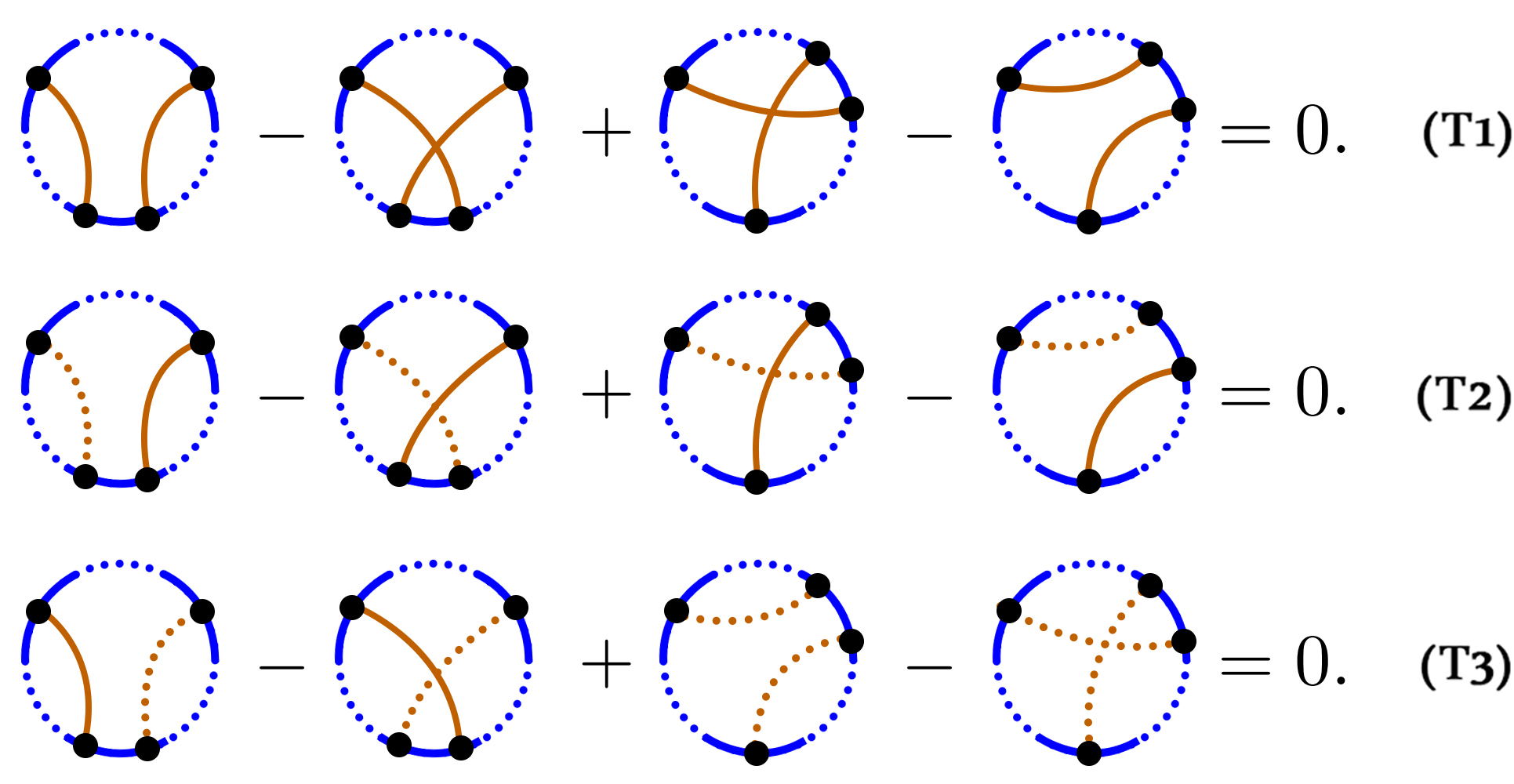}
  \caption{Four-term relations for framed chord diagrams}\label{p01}
\end{figure}

\begin{definition}[\cite{IDP, SKL}]
The module $\mathcal{M}^f$ of framed chord diagrams is the quotient module of $M^f$ modulo the relations shown in Figure \ref{p01}, called four-term relations.
\end{definition}

\begin{remark}
As usual, in these relations, as shown in Figure \ref{p01}, it is assumed that, apart from the chords shown, all framed chord diagrams in the relations may contain some other chords, which
connect points in the dashed parts of the core circle and are the same for all the four diagrams. For more detailed discussions of the framed chord diagrams and four-term relations, see \cite{IDP, MKA, SKL}.
\end{remark}

Suppose that each chord consists of two half-chords in a framed chord diagram. A signed rotation of a framed chord diagram is a cyclic order of the half-chords at the core circle induced by the orientation at that core circle, and if the chord is orientable, then we give the same sign ($+$ or $-$) to both two half-chords and give the different signs (one $+$, the other $-$) otherwise. The sign $+$ is always omitted. See Figure \ref{p06} for an example. Reversing the signs of both half-chords with a given chord yields an equivalent signed rotation. Note that the signed rotation of a framed chord diagram is independent of the choice of the first alphabet. Conversely, given a signed rotation, a framed chord diagram results from creating a circle, giving the circle a counterclockwise orientation, and labelling the cyclic order of the half-chords to the circle. We connect a solid curve for the same alphabet and sign, and a dashed curve for the same alphabet and different signs. Let $P=(p_1, p_2, \cdots, p_k)$  be a string, i.e., a finite sequence of letters from some finite alphabet. The inverse of $P$ is the string $P^{-1}=(-p_k, \cdots, -p_2, -p_1)$.
\begin{figure}[htbp!]
  \centering
  \includegraphics[width=15cm]{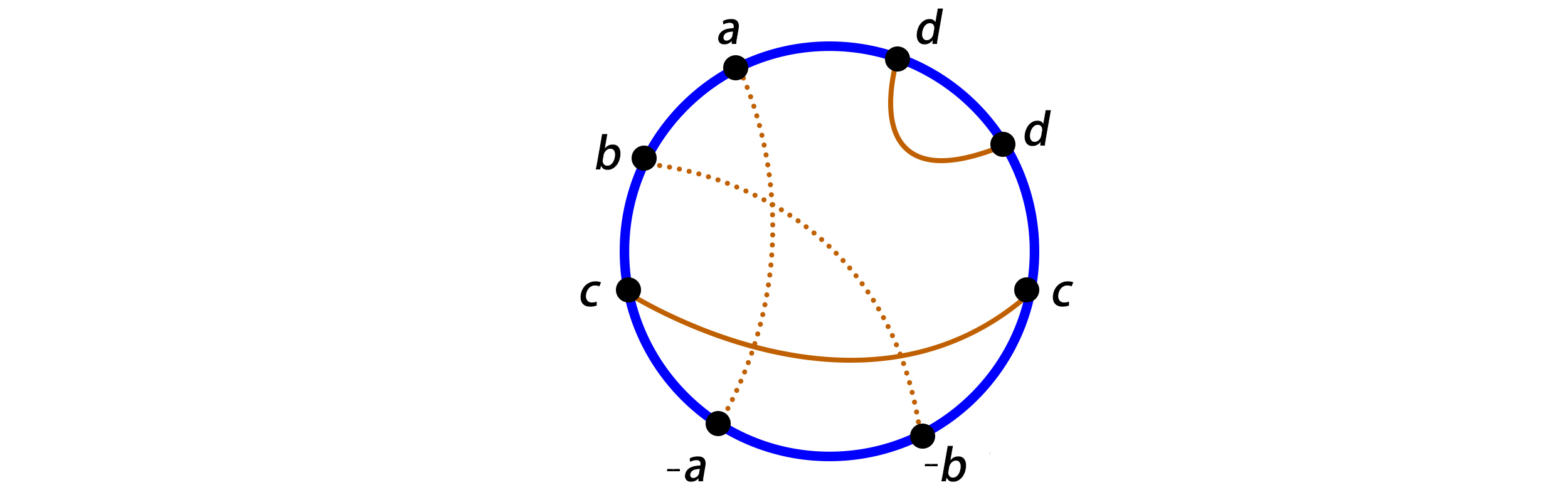}\\
  \caption{A framed chord diagram for $(a, b, c, -a, -b, c, d, d)$}\label{p06}
\end{figure}

\section{Ribbon graphs and partial-duals }
\begin{definition}[\cite{BR}]
A {\it ribbon graph} $G$ is a (orientable or non-orientable) surface with boundary,
represented as the union of two sets of topological discs, a set $V(G)$ of vertices, and a set $E(G)$ of edges,
subject to the following restrictions.
\begin{enumerate}
\item[(1)] the vertices and edges intersect in disjoint line segments, we call them {\it common line segments} as in \cite{Metrose};
\item[(2)] each of such common line segments lies on the boundary of precisely one vertex and precisely one edge;
\item[(3)] Every edge contains exactly two such common line segments.
\end{enumerate}
\end{definition}

The genus of a ribbon graph $G$, denoted by $\gamma(G)$, is its genus when viewed as a punctured surface. The Euler genus $\varepsilon(G)$ of a ribbon graph $G$ is defined as follows. If $G$ is a connected, then
\begin{eqnarray*}
\varepsilon(G)=\left\{\begin{array}{ll}
                      2\gamma(G), & \mbox{if}~G~\mbox{is orientable,}\\
                       ~~\gamma(G), & \mbox{if}~G~\mbox{is non-orientable.}
                   \end{array}\right.
\end{eqnarray*}
If $G$ is not connected, then $\varepsilon(G)$ is defined as the sum of the values of $\gamma(G_i)$ of all components $G_i$ of $G$.

For a ribbon graph $G$ and a subset $A$ of its edge-ribbons $E(G)$,  the \emph{partial dual} \cite{CG} $G^A$  of $G$ concerning $A$ is the ribbon graph obtained from $G$ by glueing a disc to $G$ along each boundary component of the spanning ribbon subgraph $(V (G), A)$ (such discs will be the vertex-discs of $G^A$), removing the interiors of all vertex-discs of $G$ and keeping the edge-ribbons unchanged.

\begin{table}[htbp!]
\caption{The partial dual of an edge of a ribbon graph}
  \centering
  \includegraphics[width=16cm]{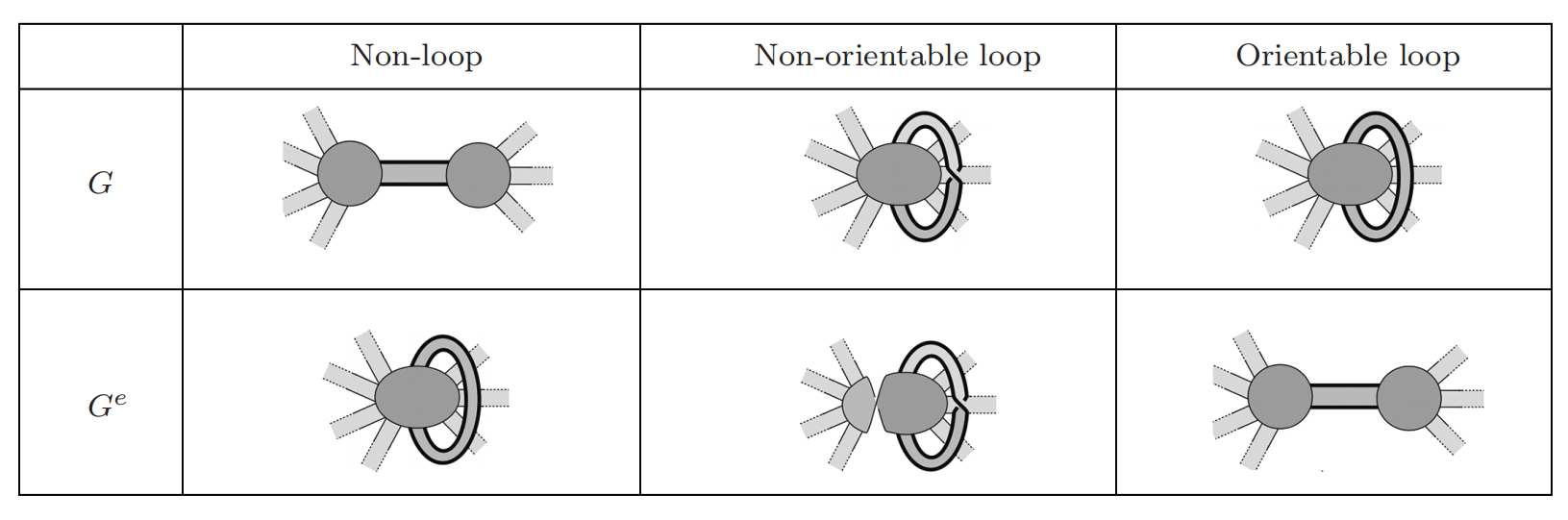}\\

  \label{p09}
\end{table}

The local effect of forming the partial dual with respect to an edge of a ribbon graph is shown in Table 1 \cite{EMO}. It is easy to check that $(G^e)^f=(G^f)^e$ for $e, f\in E(G)$ and $(G^{e})^e=G$. Then the non-loop and the orientable loop cases are reciprocal. Furthermore, for $A\subseteq E(G)$, $G^A$ can also be the ribbon graph obtained from $G$ by forming the partial dual with respect to each edge of $A$ in any order.

\begin{definition}[\cite{GMT}]\label{def-1}
The \emph{partial-dual polynomial} of any ribbon graph $G$ is defined to be the generating function $$^{\partial}\varepsilon_{G}(z)=\sum_{A\subseteq E(G)}z^{\varepsilon(G^{A})}$$
that enumerates all partial duals of $G$ by Euler genus.
\end{definition}

In the following, we define analogues of partial duals and partial-dual polynomials for framed chord diagrams. With each framed chord diagram $D$ on $n$ circles, we can associate a ribbon graph of $D$, denoted by $R(D)$, by attaching $n$ discs to the circles of $D$ and thickening the solid and dashed chords of $D$ to regular bands and half-twist bands, respectively. Note that the structure of the ribbon graph corresponding to a framed chord diagram does not depend on the orientation of the circles of the diagram. An example of framed chord diagrams and their corresponding ribbon graphs are given in Figure \ref{p14}.

\begin{figure}[htbp!]
  \centering
  \includegraphics[width=16cm]{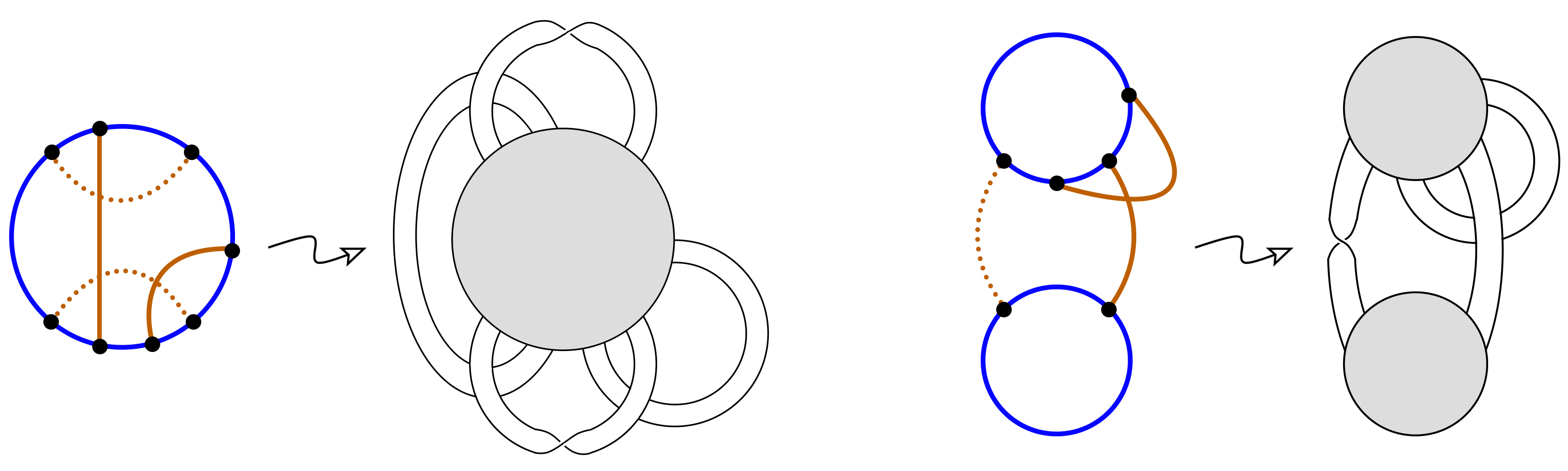}\\
  \caption{Framed chord diagrams and their corresponding ribbon graphs}
  \label{p14}
\end{figure}

\begin{remark}\label{re01}
In general, a framed chord diagram and its mirror reflection are different, but their corresponding ribbon graphs are equivalent. Because ribbon graphs are considered up to equivalence under vertex flips, i.e., reversing the cyclic order at a vertex and simultaneously giving the half-twists on the incident edges (twice in the case of a loop).
\end{remark}

\begin{lemma}[\cite{CHS}]\label{le02}
Sliding one edge-ribbon of a ribbon graph along another one results in a homeomorphic surface. Thus they have equal genera.
\end{lemma}
\begin{figure}[htbp!]
  \centering
  \includegraphics[width=12cm]{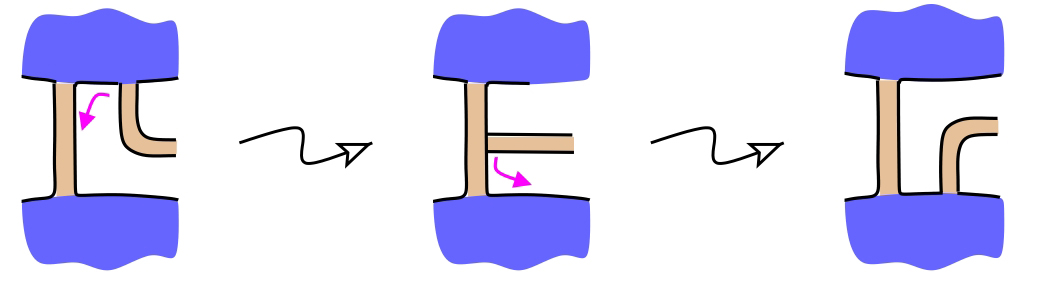}\\
\end{figure}

It is not difficult to interpret sliding and flipping operations on a framed chord diagram $D$. Some examples are given in Figure \ref{p07}.
\begin{figure}[htbp!]
  \centering
  \includegraphics[width=16cm]{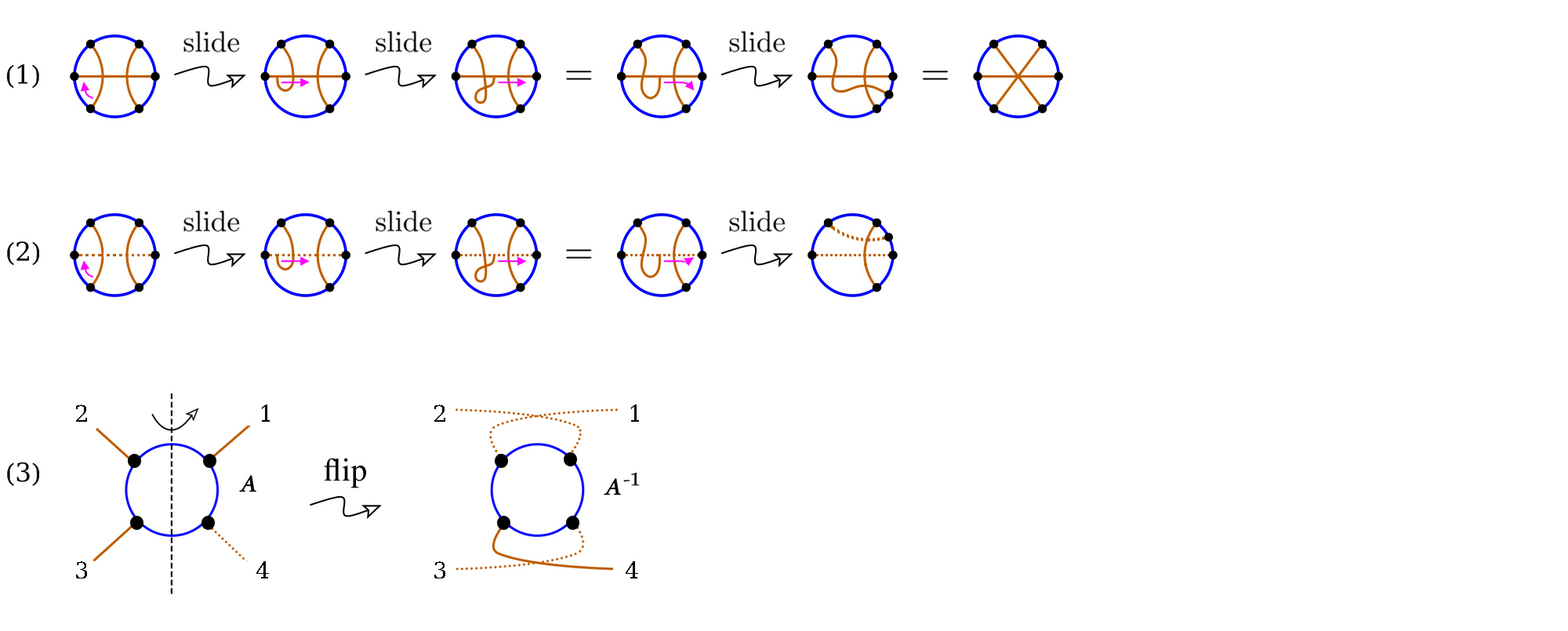}\\
  \caption{An example of sliding and flipping operations}
  \label{p07}
\end{figure}

By Lemma \ref{le02} and Remark \ref{re01}, we see that the following lemma holds:

\begin{lemma}\label{le03}
Let $D$ be a framed chord diagram. A sliding or a flipping operation on $D$ does not change the Euler genus of the ribbon graph $R(D)$.
\end{lemma}

In order to describe partial duality in the language of framed chord diagrams, we need to encode multi-vertex ribbon graphs by framed chord diagrams as well. This can be done using framed chord diagrams on several circles. Since any chord may be connected by a solid or dashed chord to one or two circles, the local effect of forming the partial dual with respect to a chord of a framed chord diagram on several circles has four cases as shown in Figure \ref{p08}.
\begin{figure}[ht!]
  \centering
  \includegraphics[width=16cm]{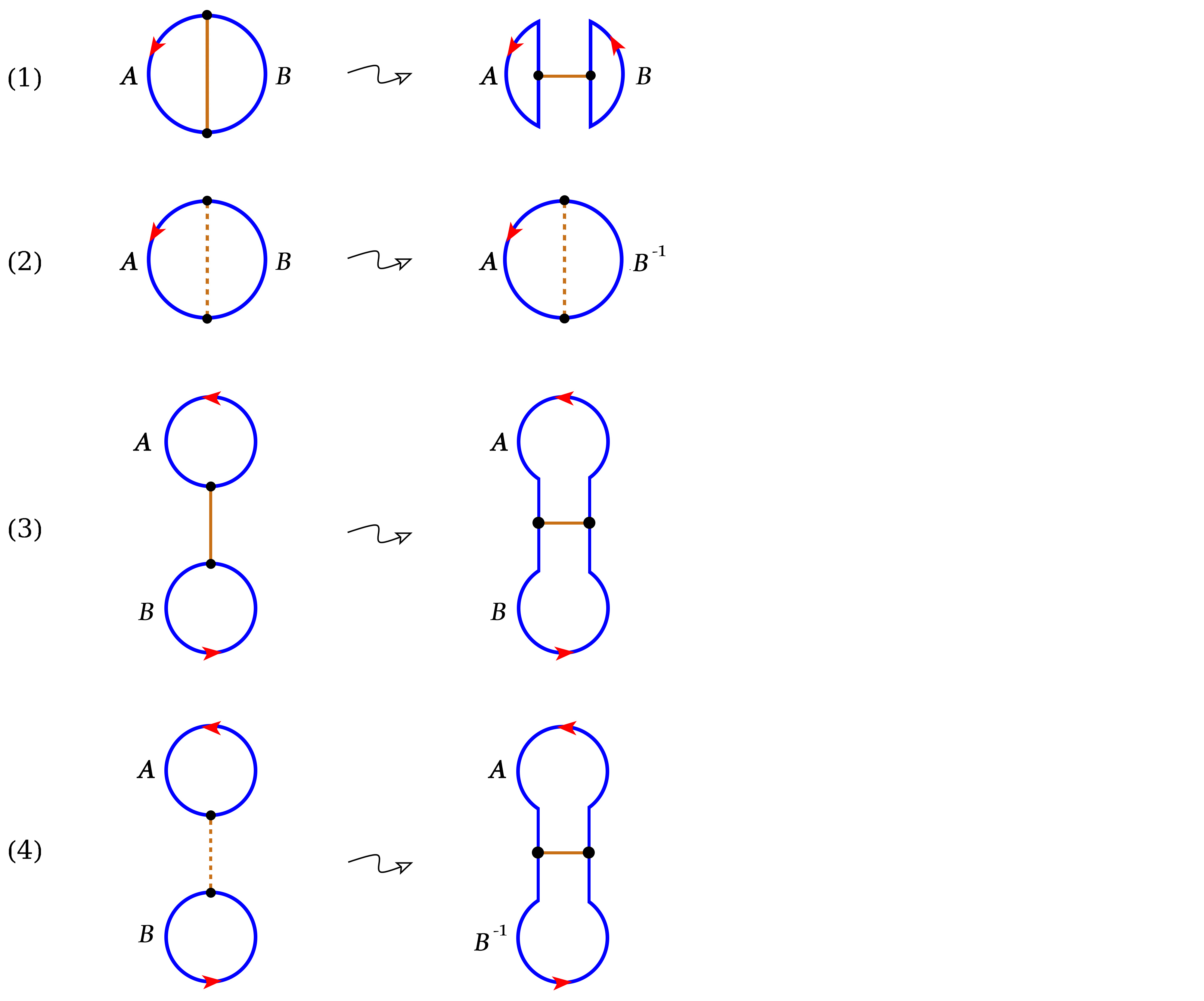}\\
  \caption{The partial dual of a chord of a framed chord diagram}
  \label{p08}
\end{figure}

\begin{remark}
The partial dual of a chord of a framed chord diagram is illustrated locally at a chord in Figure \ref{p08}. Moreover, cases 1 and 3 are reciprocal.
\end{remark}

\begin{definition}
The \emph{partial-dual polynomial} of any framed chord diagram $D$ is defined  to be $^{\partial}\varepsilon_{R(D)}(z).$
\end{definition}

\section{Main Result}
\begin{theorem}[\cite{CHS}\label{th01}]
The partial-dual polynomial as a function on chord diagrams satisfies equation $(T_1)$ of the four-term relations.
\end{theorem}

\begin{theorem}
The partial-dual polynomial as a function on framed chord diagrams satisfies the four-term relations.
\end{theorem}

\begin{proof}
The first equation (T$_1$) follows directly from Theorem \ref{th01}. We now verify the second equation (T$_2$). Each of the four framed chord diagrams of (T$_2$) contains two specific chords. Forming all possible partial duals relative to these two chords of the four framed chord diagrams of (T$_2$) gives a set of 16 framed chord diagrams in Figure \ref{p04}.
\begin{figure}[ht]
  \centering
  \includegraphics[width=16cm]{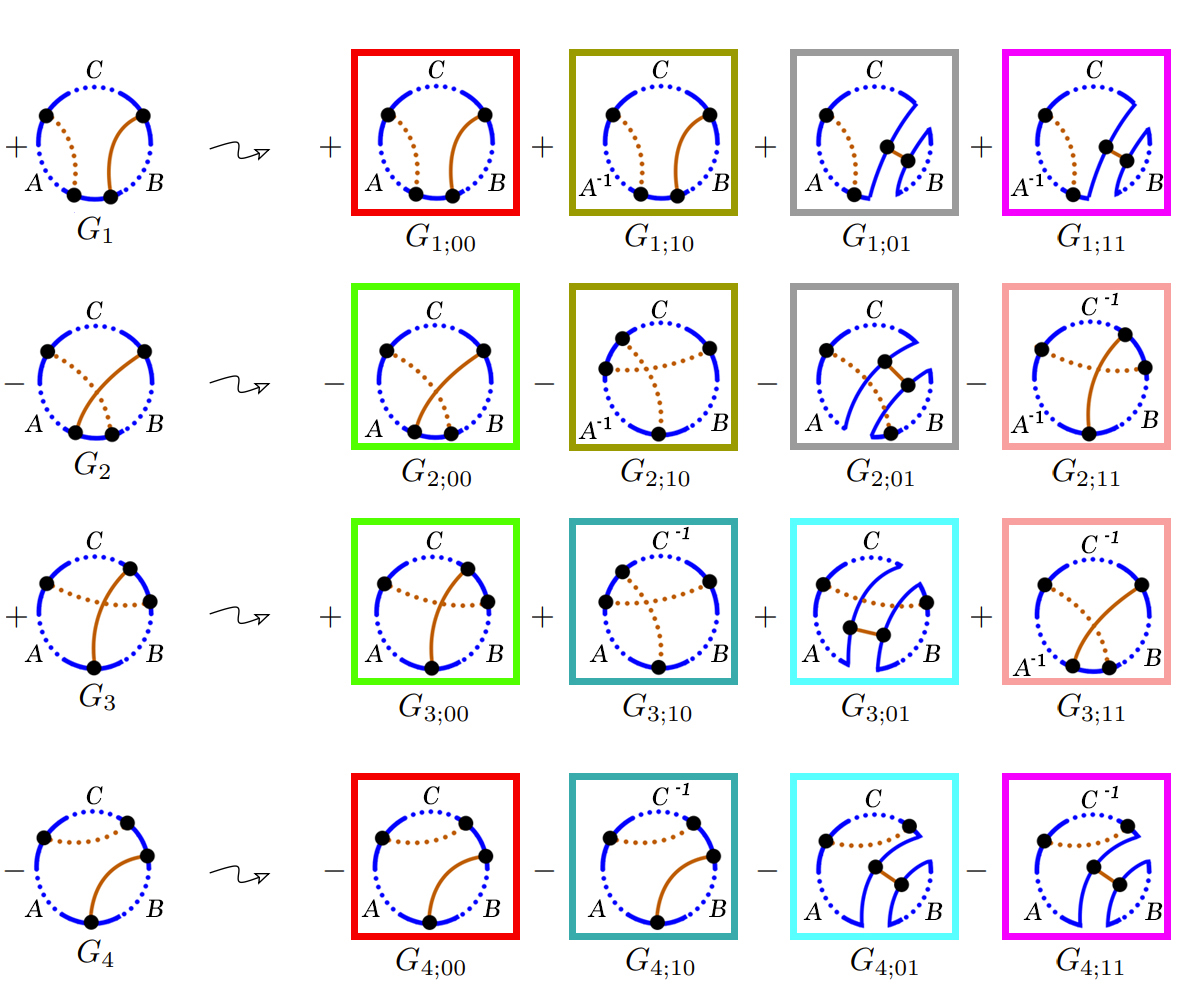}\\
  \caption{The proof of (T$_2$)}\label{p04}
\end{figure}
We partition these framed chord diagrams into 8 pairs by different coloured frames.
For convenience, we denote by $G_{1;10}$ the result of partial duality of the first diagram of (T$_2$) relative to the dashed chord, by $G_{2;01}$ the second diagram of (T$_2$) relative to solid chord, and by $G_{3;11}$ the third diagram of (T$_2$) relative to both chords. A similar convention applies to other framed chord diagrams. Note that the partial duals for the chords not indicated in (T$_2$) are the same for all four framed chord diagrams. By Lemma \ref{le03}, sliding and flipping operations do not change the Euler genera of the surfaces corresponding to the framed chord diagrams, and therefore cancel each other in the partial-dual polynomial. It is sufficient to show that the framed chord diagrams in each pair can be obtained from one to another by a sequence of sliding and flipping operations.

For the framed chord diagrams in the pairs $(G_{1;00}, G_{4;00}), (G_{2;00}, G_{3;00})$ and $(G_{2;11}, G_{3;11})$ one can be obtained from another by sliding the dashed chord along the solid chord. Analogously, $G_{2;10}$ and $G_{3;10}$ can be obtained from $G_{1;10}$ and $G_{4;10}$ by sliding the solid chord along the dashed chord, respectively. Here are the sliding operations for the remaining 3 pairs.

$(G_{1;01}, G_{2;01}):$
\begin{figure}[ht!]
  \centering
  \includegraphics[width=10cm]{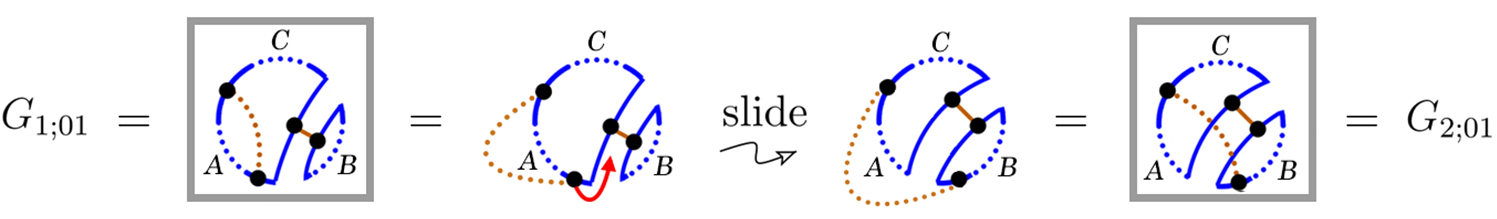}\\
\end{figure}

The sliding in the pair $(G_{3;01}, G_{4;01})$  is very similar. The last pair $(G_{1;11}, G_{4;11})$ of the last column can be done as follows.

$(G_{1;11}, G_{4;11}):$
\begin{figure}[ht!]
  \centering
  \includegraphics[width=10cm]{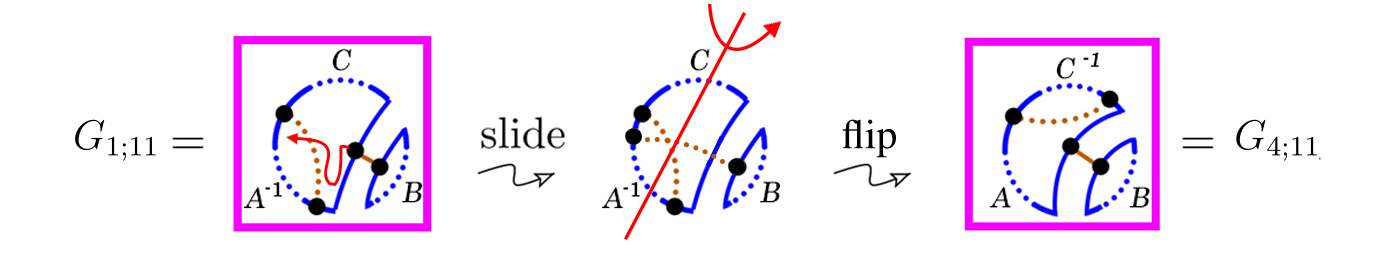}\\
\end{figure}

The same method of proof can therefore be used to establish the equation (T$_3$), as shown in Figure \ref{p05}.
\end{proof}

\begin{figure}[ht!]
  \centering
  \includegraphics[width=10cm]{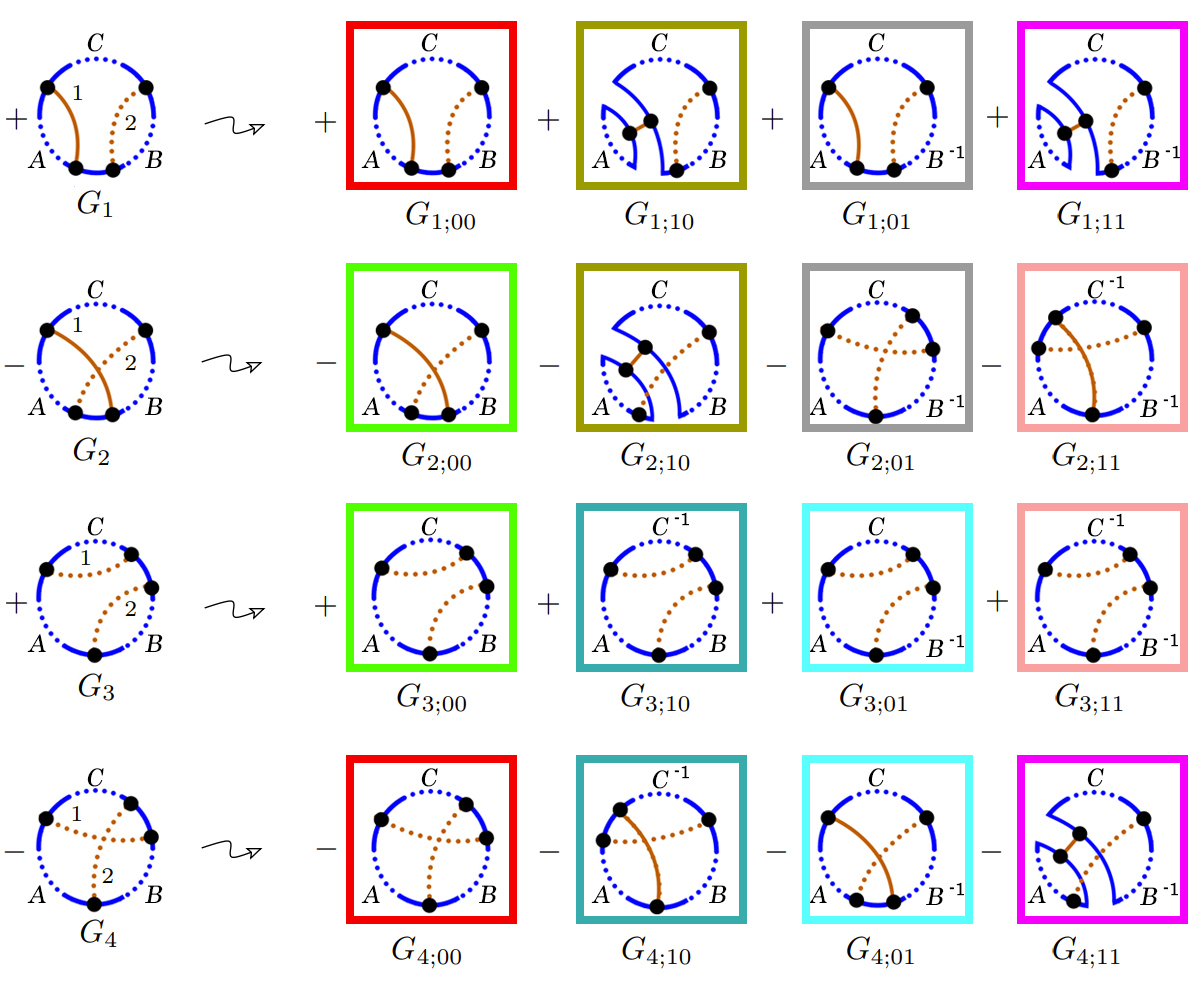}\\
  \caption{The proof of (T$_3$)}\label{p05}
\end{figure}

\subsection*{Acknowledgements}

This work is supported by NSFC (Nos. 12101600, 12171402). This paper was initiated during Xian'an Jin's visit in August 2023 at the National Institute of Education, Nanyang
Technological University. His visit was financially supported by the Ministry of Education, Singapore, under its Academic Research Tier 1 (RG19/22). The work was partially supported by the Natural Science Foundation of Hunan Province, PR China (No. 2022JJ40418), the Excellent Youth Project of Hunan Provincial Department of Education, PR China (No. 23B0117), the China Scholarships Council (No. 202108430063). We thank the National Institute of Education, Nanyang Technological University, where part of this research was performed.

\EditInfo{February 8, 2024}{March 6, 2024}{Jacob Mostovoy and Sergei Chmutov}


\begin{thebibliography}{}

\end{thebibliography}


{\small
    \begin{thebibliography}{10}
    
    \bibitem{BR}
    B.~Bollob\'{a}s and O.~Riordan.
    \newblock A polynomial of graphs on surfaces.
    \newblock {\em Math. Ann.}, 323(1):81--96, 2002.
    
    \bibitem{CG}
    S.~Chmutov.
    \newblock Generalized duality for graphs on surfaces and the signed {B}ollob\'{a}s-{R}iordan polynomial.
    \newblock {\em J. Combin. Theory Ser. B}, 99(3):617--638, 2009.
    
    \bibitem{CHS}
    S.~Chmutov.
    \newblock Partial-dual genus polynomial as a weight system.
    \newblock {\em Commun. Math.}, 31(3):113--124, 2023.
    
    \bibitem{CHVT}
    S.~Chmutov and F.~Vignes-Tourneret.
    \newblock On a conjecture of {G}ross, {M}ansour and {T}ucker.
    \newblock {\em European J. Combin.}, 97:103368, 2021.
    
    \bibitem{EMO}
    J.~A. Ellis-Monaghan and I.~Moffatt.
    \newblock {\em Graphs on Surfaces: Dualities, Polynomials, and Knots}.
    \newblock Springer New York, 2013.
    
    \bibitem{GMT}
    J.~L. Gross, T.~Mansour, and T.~W. Tucker.
    \newblock Partial duality for ribbon graphs, {I}: distributions.
    \newblock {\em European J. Combin.}, 86:103084, 2020.
    
    \bibitem{GMT2}
    J.~L. Gross, T.~Mansour, and T.~W. Tucker.
    \newblock Partial duality for ribbon graphs, {II}: partial-twuality polynomials and monodromy computations.
    \newblock {\em European J. Combin.}, 95:103329, 2021.
    
    \bibitem{IDP}
    D.~P. Ilyutko and V.~O. Manturov.
    \newblock A parity map of framed chord diagrams.
    \newblock {\em J. Knot Theory Ramifications}, 24(13):1541006, 2015.
    
    \bibitem{MKA}
    M.~Karev.
    \newblock On the primitive subspace of {L}ando framed graph bialgebra.
    \newblock {\em Commun. Math.}, 31(3):125--136, 2023.
    
    \bibitem{SKL}
    S.~K. Lando.
    \newblock {$J$}-invariants of plane curves and framed chord diagrams.
    \newblock {\em Funct. Anal. Appl.}, 40(1):1--10, 2006.
    
    \bibitem{Metrose}
    M.~Metsidik.
    \newblock {\em Characterization of some properties of ribbon graphs and their partial duals}.
    \newblock PhD thesis, Xiamen University, 2017.
    
    \bibitem{YJ1}
    Q.~Yan and X.~Jin.
    \newblock Counterexamples to a conjecture by {G}ross, {M}ansour and {T}ucker on partial-dual genus polynomials of ribbon graphs.
    \newblock {\em European J. Combin.}, 93:103285, 2021.
    
    \bibitem{YJ2}
    Q.~Yan and X.~Jin.
    \newblock Counterexamples to the interpolating conjecture on partial-dual genus polynomials of ribbon graphs.
    \newblock {\em European J. Combin.}, 102:103493, 2022.
    
    \bibitem{YJ3}
    Q.~Yan and X.~Jin.
    \newblock Partial-dual genus polynomials and signed intersection graphs.
    \newblock {\em Forum Math. Sigma}, 10:e69, 2022.
    
    \end{thebibliography}
}
\end{document}